\documentclass[12pt]{amsart}
\usepackage{amsmath}
\usepackage{amsfonts}
\usepackage{amssymb}
\usepackage{amsthm}
\usepackage[all]{xy}

\newtheorem{theorem}{Theorem}[section]
\newtheorem{lemma}[theorem]{Lemma}
\newtheorem{proposition}[theorem]{Proposition}
\newtheorem{corollary}[theorem]{Corollary}
\newtheorem{definition}[theorem]{Definition}
\newtheorem{remark}[theorem]{Remark}

\numberwithin{equation}{section}

\begin{document}

\baselineskip=15pt

\title{Picard bundles and Twisted Picard bundles on the Jacobian of a curve} 

\author[Usha]{Usha N. Bhosle}

\address{ Tata Institute of Fundamental Research, Mumbai- 400005, India}

\email{usnabh07@gmail.com}



\thanks{ This work was initiated during my tenure as an INSA Senior Scientist  at the Indian Statistical Institute, Bangalore, a part of this work was done during my visits to TIFR. I thank both the institutes for hospitality. I am grateful to Akiyoshi Sannai, Department of physics, Kyoto University for a personal correspondence  pointing out a couple of errors in the previous version (arXiv:2602.1915v1). The errors did not affect the rest of the paper and are corrected in this version.  } 


\date{20 October 2026}

\begin{abstract}
  Let $Y$ denote an irreducible projective curve with at most nodes as singularities and defined over  an algebraically closed field of characteristic zero. We study the restriction of the twisted Picard bundles on the compactified Jacobian $\overline{J}(Y)$ of $Y$ to the embedded curve in $\overline{J}(Y)$. As an application, we show that for $g =2$ and each integer $r \ge 3$, there is a two-dimensional family of stable ACM bundles on the compactified Jacobian which has the Picard bundle in its limit. We define an embedding $\alpha_Y$ of the (generalised) Jacobian $J(Y)$ in the moduli space $U^s_Y(n,d)$ of stable vector bundles on $Y$ using a twisted restriction $E_Y$ of a Picard bundle to $Y \subset J(Y)$. We show that (under suitable conditions) the restriction of the universal bundle $\mathcal{U}$ to $Y \times J(Y)$ is stable for suitable polarisation. For the embedding of a smooth curve $Y$ given by $E_Y \otimes B, B$ a line bundle of degree $b$,  we show that the restriction of the Picard bundle on $U^s_Y(n, d+nb)$ to $J(Y)$ is $\theta$-semistable for $b \ge 2g-1$ and $\theta$-stable for $b \ge 2g$. We also determine the relation between the restriction of the theta divisor on $U^s_Y(n,d+nb)$ to $J(Y)$ and the theta divisor $\theta$ on $\overline{J}(Y)$.         
\end{abstract}

\maketitle

\section{Introduction}

                    Let $Y$ denote an irreducible projective curve with at most nodes as singularities and defined over  an algebraically closed field of characteristic $0$. Fix a nonsingular point $t \in Y$. Let $J(Y)= Pic^0(Y)$ be its generalised Jacobian and $\overline{J}(Y)$ the compactified Jacobian. There is a Poincar\'e sheaf $\mathcal{P} \to \overline{J}(Y) \times \overline{J}(Y)$ which restricts to the (normalised) Poincar\'e sheaf $\mathcal{L} \to \overline{J}(Y) \times Y$. For $r \ge 2g-1$, the  Picard  bundle $E_r$ on $\overline{J}(Y)$ is the direct image of $\mathcal{L} \otimes p_Y^*\mathcal{O}_Y(r t)$ on $\overline{J}(Y)$. The Picard bundles and their restrictions to the curve $Y$ embedded in $\overline{J}(Y)$ have been studied extensively (\cite{EL}, \cite{K}, \cite{M1}, \cite{M2}, \cite{Mu}, \cite{BhP2}, \cite{BhS}, to name a few). 
                    
                    In this small note, we first study the twisted Picard bundles $V_{r,x} := E_r \otimes \mathcal{P}_x$ where $\mathcal{P}_x := \mathcal{P}\vert_{\overline{J}(Y) \times x},   x \in J(Y)$.  Considering $Y$ embedded in $\overline{J}(Y)$ using a line bundle of degree $g-1$, we study the restrictions of $V_{r,x}$ to the curve $Y \subset \overline{J}(Y)$.  The bundles $V_{r,x}$ are stable with respect to the theta line bundle for $g \ge 1, r \ge 2g -1$ (by \cite[Theorem 1.5]{BhP2}), their restrictions to $Y$ are also stable (See \cite[Sections 8 and 9]{BhP2}).  
 We  determine some cohomologies of $V_{r,x}(j \theta), j \in \mathbb{Z}$, where $\theta$ denotes the ample theta line bundle on the compactified Jacobian.
  
   We remark that the bundles $V_{r,x}, r \ge 2g-1, x^* \in Y$ are not ACM bundles as they have non vanishing first cohomology by Theorem \ref{chomB} ($x^*$ denotes the Fourier-Mukai transform of $x$).  In particular the Picard bundles are not ACM bundles. For $g =2$, we compute the cohomology of the restrictions of the Poincar\'e bundle $\mathcal{P}_x$ with ${\mathcal{P}_x}\vert_Y \neq \mathcal{O}_Y$. 
\begin{theorem} (Theorem \ref{PxACMg2}) 
        Assume that $g=2$ and $\mathcal{P}_x\vert_Y \neq \mathcal{O}_Y, \ \mathcal{P}_x$ being the restriction of the Poincar\'e bundle (defined by \eqref{Px}). Then \\
(1) $h^1(\overline{J}(Y), \mathcal{P}_x(j \theta)) = 0\, , \ \forall j \in \mathbb{Z}\, .$ In particular,  $\mathcal{P}_x$ is an ACM bundle on $\overline{J}(Y)$. \\
(2) $h^0(\overline{J}(Y), \mathcal{P}_x(j \theta)) = 0, \ \forall j \le 0\ ,$  \\
\ \  $h^0(\overline{J}(Y), \mathcal{P}_x(j \theta)) = j^2, \ \forall j \ge 0\ .$  \\ 
(3) $h^2(\overline{J}(Y), \mathcal{P}_x(j \theta)) = 0, \ \forall j \ge 0\ ,$\\
~~ $h^2(\overline{J}(Y), \mathcal{P}_x(j \theta)) = j^2, \ \forall j \le 0\ .$\\
In particular, $\mathcal{P}_x$ is not a Ulrich bundle.
\end{theorem}   

For $g=2$, we have the following result for the twisted Picard bundles.   
\begin{theorem} (Theorem \ref{thm1})  
          Let $g =2$. Then the twisted Picard bundles  $V_{r, x}, r \ge 3, x^* \notin Y$ are ACM bundles for $x$ general in $J(Y)$. They are not Ulrich bundles. 
\end{theorem}        

    We note that if $p_1, p_2$ are the projections from $\overline{J}(Y) \times {J}(Y)$ to the first and second factors, then $(p_1^*E_r) \otimes \mathcal{P}$ is the family of the stable bundles $V_{r,x}$ on $\overline{J}(Y)$ parametrised by $J(Y)$ whose generic members are ACM bundles ($g=2$). Some members in the family, including the Picard bundle, are not ACM bundles.
       
   It will be interesting to know if the results of these two theorems are true for higher genera as well.

        In the second part of the paper, we define an embedding $\alpha_Y$ of the (generalised) Jacobian $J(Y)$ in the moduli space $U^s_Y(n,d)$ of stable vector bundles on $Y$ using a twisted restriction $E_Y$ of a Picard bundle to $Y \subset J(Y)$ (see Subsection \ref{alphaembed} for details).  On smooth curves,  Yingchen Li had defined such an embedding using a stable vector bundle $E_0$ such that the  bundle of endomorphism $End (E_0)$  has no line bundle direct summands \cite{Li}. In \cite{Bh}, I had given  an embedding of the compactified Jacobian of a nodal curve into $U^s_Y(n,d)$ under the additional assumption that the pull back of $E_0$ to the normalisation $X$ is stable. The assumption on $End(E_0)$ is crucially used to prove the injectivity of $\alpha_Y$. We do not know if $E_Y$ has any direct summand which is a line bundle.  Moreover,  the pull back of $E_Y$ to the normalisation $X$ is not semistable.  Consequently the embedding $\alpha_Y$ does not extend to $\overline{J}(Y)$. For $Y$ smooth, we study the relation between the restriction of the theta divisor $\theta_U$ on $U^s_Y(n,d)$ to $J(Y)$ with the theta divisor $\theta$ on $J(Y)$. 
        
        We show that (under suitable conditions) the restriction of the universal bundle $\mathcal{U}$ to $Y \times J(Y)$ is stable for any polarisation of the form $a_1 H_1+ a_2 H_2$, where $a_1, a_2$ are positive integers,  $H_1$ and $H_2$ are any polarisations on $Y$ and $J(Y)$ respectively  (Proposition \ref{univbundleres}). 
        
      Assume that $Y$ is a smooth curve.  Let $r_Y$ and $d_Y$ be the rank and degree of $E_Y$. Let $b \ge 2g-1$ be an integer. Then we show that 
the restriction of the Picard bundle $E_{r_Y, d_Y+b r_Y}$ to $J(Y)$ is $\theta$-semistable for $b \ge 2g-1$ and $\theta$-stable for $b \ge 2g$ (Theorem  \ref{picardbundleres}).
      
                    After this paper was written, we learnt about a paper of Kota Yoshioka where he studied   ACM bundles on abelian surfaces of Picard number $1$. He determines necessary and sufficient conditions on the Mukai vector (i.e., on the rank, first Chern class and Euler characteristic of a bundle) for existence of an ACM bundle. He does not study explicitly twisted Picard bundles $V_{r,x}$.   Even in the smooth case, the Jacobian can have any Picard number between  $1$ and $4$, for examples see \cite[10.8 (7), page 312]{BL}, \cite{B}.  The compactified Jacobian is not an Abelian variety, it is not even normal.

\section{Preliminaries }

\subsection{Notation}  \hfill

      Let $Y$ denote an irreducible projective curve with at most nodes as singularities and defined over  an algebraically closed field of characteristic $0$. 
 Let $y_1, \cdots, y_k$ be the nodes of $Y$ and  $Y_0$ be the set of smooth points of $Y$.  
Let 
$$p : X \to Y$$ 
be the normalisation map.

      Let $g = g(Y) := h^1(Y, \mathcal{O}_Y)$ be the arithmetic genus of $Y$. 
Since $Y$ is a Gorenstein curve, it has a locally free dualising sheaf $\omega_Y$. 
Fix a base point $t \in Y$ different from all the  nodes 
$y_j, j=1,\cdots,k$ and sufficiently general. 

    For a vector bundle $E$ on $Y$, let $r(E)$ and $d(E)$ denote the rank and degree of $E$. Let $h^i(E) = $ dim $H^i(Y, E)$ for all $i$.

\subsection{The Compactified Jacobian}  \hfill   

          Let $J(Y)$ be the space of line bundles (locally free sheaves of rank $1$) of degree $0$ on $Y$. If $Y$ is nonsingular, it is called the Jacobian of $Y$, it is an Abelian variety. If $Y$ has nodes, $J(Y)$ is called the generalised Jacobian of $Y$, it is a non-singular quasi projective variety. Then the variety $J(Y)$ has a natural compactification 
$\overline {J}(Y)$ as the variety of torsion-free sheaves of rank $1$ and  
degree $0$ on $Y$. The variety $\overline {J}(Y)$ is not normal, it is a seminormal projective variety. 

              There is a canonically defined ample theta divisor $\theta$ on $\overline{J}(Y)$ (using $t$).   

                   There is  the Abel-Jacobi embedding 
           $$Y  \ \hookrightarrow  \ \overline{J}(Y)$$ 
given by $x \mapsto I^*_x \otimes \mathcal{O}_Y (-t)$, where $I_x$ denotes the ideal sheaf of $x \in Y$. We shall identify $Y$ with its image in  $\overline{J}(Y)$ under this embedding.

\subsection{Picard sheaves and twisted Picard bundles on $\overline{J}(Y)$} \hfill  \label{picardsheaves}
    
    There is a Poincar\'e sheaf (universal sheaf) ${\mathcal L}$ on $\overline{J}(Y)
\times Y$ such that for a point $L \in \overline{J}(Y)$, one has ${\mathcal L} \vert_{L\times Y}$ is isomorphic to the torsionfree sheaf $L$ of rank $1$  corresponding to the point $L$. The sheaf 
${\mathcal L}$ is unique up to tensoring by a line bundle on $\overline{J}(Y)$. Hence we normalise ${\mathcal L}$ by the condition that ${\mathcal L}\vert _{\overline{J}(Y)\times t} \cong {\mathcal O}_{\overline{J}(Y)}$, the trivial line bundle on $\overline{J}(Y)$. 
Let 
$p_J: \overline{J}(Y)\times Y \to \overline{J}(Y) \ {\rm and} \  
{p_Y}:\overline{J}(Y) \times Y \to Y$ 
denote the projections. Let 
$${\mathcal L}(r) = \mathcal{L} \otimes p_Y^*{\mathcal O}_{Y}(r t)\, .$$ 

\begin{definition}
We define the {\bf Picard sheaves} on $\overline{J}(Y)$ as the direct image sheaves 
$$ E_{r} := {p_J}_*({\mathcal L}(r)). $$
       For $r \ge 2g-1$,  the Picard sheaf is a vector bundle of rank $r+1-g$ on $\overline{J}(Y)$ called  a {\bf Picard bundle}.
\end{definition}
       
  Henceforth we assume that $r \geq 2g -1$. 

The Poincar\'e sheaf  ${\mathcal L}$ on $\overline{J}(Y) \times Y$ is the restriction of the Poincar\'e 
 sheaf $\mathcal{P}$ on  $\overline{J}(Y) \times \overline{J}(Y)$. This defines a Fourier-Mukai transform $\overline{J}(Y) \to \overline{J}(Y)$ which is an isomorphism (\cite{Mu} for $Y$ smooth, \cite{BhS} for $Y$ nodal). 

       For $x \in {J}(Y)$, let $x^*$ denote the Fourier-Mukai transform of $x$. Let $i_x: \overline{J}(Y)  \hookrightarrow \overline{J}(Y) \times \overline{J}(Y), i_x(L) =  (L, x)$ be the inclusion. 
Define {\bf the restriction of the Poincar\'e bundle}  
\begin{equation} \label{Px}
\mathcal{P}_x : = i_x^* \mathcal{P}\, .
\end{equation}
The cohomology of $P_x$ is known.
\begin{theorem} \label{cohomPx} 
(\cite[Theorem B]{A} for $Y$ with planar singularities; \cite{Mu} for $Y$ smooth).\\
(1) If $ x \in J(Y), x \neq \mathcal{O}_Y$,
 $$H^i(\overline{J}(Y), P_x) = 0 \  \forall \ i \, .$$
(2) For $x = \mathcal{O}_Y$, we have $P_x = \mathcal{O}_{\overline{J}_Y}$.  Then $h^0(\overline{J}(Y), \mathcal{O}_{\overline{J}_Y}) = 1$ and for $1\leq i \leq g$\ ,   
$$H^i(\overline{J}(Y), \mathcal{O}_{\overline{J}_Y}) \ \cong \ \bigwedge^i H^1(Y, \mathcal{O}_Y) \, $$
so that
$$h^i(\overline{J}(Y), \mathcal{O}_{\overline{J}_Y})= \binom{g-1}{i-1} \, .$$
\end{theorem}

\begin{definition} \label{Vrx}
For $r \ge 2g - 1$ and $x \in J(Y)$, define the vector bundles 
$$V_{r,x} := E_r \otimes \mathcal{P}_x\, .$$
We call $V_{r,x}$ a {\bf twisted Picard bundle}. 
\end{definition}
                 
\begin{theorem} \label{chomB} (\cite[Corollary 2.13]{BhS}; \cite{Mu} for $Y$ smooth).
Let $r > g-1$.
 $$h^0(\overline{J}(Y), V_{r,x}) = 0 \  \forall \ x \in J(Y)\, .$$
 For $1\leq i \leq g$ and $x \in J(Y)$, we have 

\[ h^i(\overline{J}(Y), V_{r,x})=
\begin{cases}
 \binom{g-1}{i-1} &  \mbox{if } x^*\in Y\\ 
 0 &   \mbox{if } x^*\notin Y. 
\end{cases}
\]

\end{theorem}

\subsection{Some general results}  \hfill

\begin{lemma} \label{L'0}
Let $Z$ be an equi-dimensional  Cohen-Macaulay projective scheme of dimension $g$ with the dualising sheaf $\omega_Z$ the trivial line bundle. Let $H$ be  an ample line bundle and $E$ be a locally free sheaf on $Z$. Let $i$ be an integer, $0 \le  i \le  g$. Then: \\
(1) $H^{i}(Z, E) \cong H^{g-i}(Z, E^*)$\, .\\
(2)  $H^i(Z, E) = 0$ if and only if $H^{g-i}(Z, E^*) = 0$\, .
\end{lemma}
\begin{proof}
(1)   Since $E$ is locally free, by Serre duality \cite[Theorem 7.6]{H}, $h^{g-i}(Z, E^*) = h^i(Z, \omega_Z \otimes  E) = h^i(Z, E)$  as  $\omega_Z$ is the trivial line bundle.  This proves Part (1). \\
(2) Part (2) follows immediately from Part (1).
\end{proof}

Lemma \ref{L'0} is not true if $E$ is replaced by a reflexive sheaf, a counter example was communicated to me by Akiyoshi Sannai.
 
\begin{corollary} \label{C'1}
Let $r \ge 2g-1$. Then
 $$h^g(\overline{J}(Y), V^*_{r,x}) = 0 \  \forall \ x \in J(Y)\, .$$
 For $0\leq i \leq g-1$ and $x \in J(Y)$, we have 

\[ h^{i}(\overline{J}(Y), V^*_{r,x})=
\begin{cases}
 \binom{g-1}{g-1-i} &  \mbox{if } x^*\in Y\\ 
 0 &   \mbox{if } x^*\notin Y. 
\end{cases}
\]

\end{corollary}

\begin{proof} 
For $r \ge 2g -1$, $V_{r,x}$ is a vector bundle. The corollary follows from Theorem \ref{chomB} and Lemma \ref{L'0}.
\end{proof}

Write $V_{r,x}(j \theta)) := V_{r,x} \otimes  \theta^{\otimes j}$.

\begin{lemma} \label{L'1}
Assume that $x^*\notin Y$.
\begin{enumerate}
\item $H^0(\overline{J}(Y), V_{r,x}(j \theta)) = 0 \ \forall j \le 0\ ; \ H^g(\overline{J}(Y), V^*_{r,x}(j \theta)) = 0 \ \forall j \ge 0\ $.
\item $H^g(\overline{J}(Y), V_{r,x}(j \theta)) = 0 \ \forall j \ge 0\ ; \ H^0(\overline{J}(Y), V^*_{r,x}(j \theta)) = 0 \ \forall j \le 0\ .$
\end{enumerate}
\end{lemma}
\begin{proof}
(1) For $j \le 0$, one has $H^0(\overline{J}(Y), V_{r,x}(j \theta)) \subset H^0(\overline{J}(Y), V_{r,x}) = 0$ by Theorem \ref{chomB}. The second part of (1) follows from this by Lemma \ref{L'0}.\\
(2) By Corollary \ref{C'1}, $H^0(\overline{J}(Y), V^*_{r,x}) = 0$. Hence $H^0(\overline{J}(Y), V^*_{r,x}(j\theta)) 
\subset H^0(\overline{J}(Y), V^*_{r,x}) = 0 \ \forall j \le 0\ .$ Then $H^g(\overline{J}(Y), V_{r,x}(j \theta)) = 0 \ \forall j \ge 0$ by Lemma \ref{L'0}.
\end{proof}

\section{Restriction of the twisted Picard bundle to $Y \subset \overline{J}(Y)$} \hfill

Let $r \ge 2g -1$, let $L$ be a line bundle of degree $r-1$. Using $L$, one can give an embedding of $Y$ in $\overline{J}^r(Y)$ and hence in $\overline{J}(Y)$ using $t$ (see \cite{EL} in case $Y$ smooth, \cite{BhP2} in nodal case). For $L, L_1 \in \overline{J}^{r-1}$, the embeddings given by $L$ and $L_1$ are translates of each other. 
 There exists an exact sequence (see \cite[Equation (1.4)]{EL}, \cite[Equation (8.7)]{BhP2})
\begin{equation}\label{d4a} 
0 \to H^0(Y, L)\otimes {\mathcal O}_{Y} \to {E}_{r}\vert_{Y} (t) \to L' \to 0, 
\end{equation}
where $L' $ is a line bundle of degree $r+1-2g, \ h^0(Y, L) = r- g, \ d({E}_{r}\vert_{Y}) = - g$.\\

Define         
$$\theta_Y:= \theta\vert_Y\, ,$$
the restriction of the line bundle $\theta$ to $Y$.
 
        In the case $Y$ is smooth,  Ein and Lazarsfeld proved the stability of $E_r\vert_Y$ (\cite[Proposition 1.5]{EL}, \cite[Proposition 2.2]{EL}).  In case $Y$ has nodes, the stability of $E_r\vert_Y$ was proved by myself and Parameswaran (see \cite[Theorem 8.10]{BhP2} for non-rational curve, \cite[Theorem 9.3]{BhP2} for a rational nodal curve with at least one node).

The exact sequence \ref{d4a}  gives

\begin{equation}\label{Vry}
0\to H^0(Y, E_r\vert_Y (t)) \otimes {\mathcal O}_{Y} (-t) \otimes \mathcal{P}_x\vert_Y \longrightarrow
V_{r,x}\vert_Y  \longrightarrow L' (-t) \otimes \mathcal{P}_x\vert_Y \to 0\, . 
\end{equation}

\begin{proposition} \label{VL1}
Let $r \ge 2g -1$.
$$h^0( (V_{r,x}(j \theta))\vert_Y) \ = \ 0\, , \ \forall  j \le 0\, ,$$
$$h^1( (V_{r,x}(j \theta))\vert_Y) \ = (r+1-g)( g(1- j) - 1) + g \ , \ \forall  j \le 0\, .$$ 
\end{proposition}
\begin{proof}
                      One has  the intersection $\theta \cdot Y = g$. Hence $\theta_Y$ is a line bundle of degree $g$ on $Y$. We note that due to the normalisation of $\mathcal{P}$, the line bundle $\mathcal{P}_x$ has first Chern class $0$ and hence $\mathcal{P}_x\vert_Y$ has degree $0$ on $Y$.
Hence we have         
$$d(V_{r,x} (j \theta) \vert_Y) = d(V_{r,x} \vert_Y) + j g \  r(V_r)) = - g + j g(r+1-g) \le -g < 0$$
for $j \le 0$. Since $V_{r,x}(j \theta)\vert_Y$ is a stable vector bundle of negative degree (\cite[Sections 8 and 9]{BhP2}), it has no non-zero sections on $Y$.

The nontrivial cohomology $H^1(V_{r,x}(j \theta)\vert_Y)$ can now be computed using Riemann-Roch theorem.
\end{proof}

\begin{proposition} \label{VL2}
Let  $r \ge 2g - 1$.
Then\\
(1) $\forall j \ge 2\, ,$ one has 
$$ h^1( (V_{r,x} (j \theta)) \vert_Y) \ = \ 0\, , $$ 
$$ h^0( (V_{r,x} (j \theta)) \vert_Y) \ = \ ((j-1)g+1)(r+1-g) -g\, ,$$
(2) ${\rm for \ a \ general} \  x \in J(Y)\, ,$ we have 
$$ h^1( (V_{r,x} (\theta)) \vert_Y) \ = \ 0\, , $$
$$ h^0( (V_{r, x} (\theta))\vert_Y) = r+1-2g\, .$$
\end{proposition}

\begin{proof}
   Tensoring  the exact sequence \eqref{Vry} with $(j \theta_Y) = \theta_Y^{\otimes j}$ and taking the  cohomology exact sequence associated to it, we get the short exact sequence 
\begin{equation}  \label{jpositive}   
    H^0(Y, E_r\vert_Y (t)) \otimes H^1(j {\theta_Y} (-t) \otimes \mathcal{P}_x\vert_Y) \rightarrow H^1(V_{r,x} (j \theta) \vert_Y)  \rightarrow H^1(L' (-t) \otimes \mathcal{P}_x\vert_Y \otimes  j \theta_Y) \, .
\end{equation}

 One has 
$d(j {\theta_Y} (-t) \otimes \mathcal{P}_x\vert_Y) = j g -1$. Since $j g -1\ge 2g-1$ for $j \ge 2$, one has  
$$H^1(j {\theta_Y} (-t) \otimes \mathcal{P}_x\vert_Y) = 0 \ {\rm for}  \ j \ge 2\, .$$

Take $j =1$. Then $d(\theta_Y (-t) \otimes \mathcal{P}_x\vert_Y) = g -1$ and $h^1(\theta_Y (-t) \otimes \mathcal{P}_x\vert_Y) = 0$ for $x$ general (i.e., $x$ in an (unspecified) open subset of 
$J(Y)$). Hence  
$$H^1({\theta_Y} (-t) \otimes \mathcal{P}_x\vert_Y) = 0\, , \ {\rm for \ a \  general} \  x \in J(Y)\, .$$

Take $j=1, g=2$. Then $\omega_Y = \theta_Y$ \cite{BL}. By Serre duality, one has 
$h^1({\theta_Y} (-t) \otimes \mathcal{P}_x\vert_Y) = h^0(\omega_Y\otimes  \theta^*_Y (t) \otimes (\mathcal{P}_x\vert_Y)^*) = \mathcal{P}^*_x\vert_Y (t)= 0$ if $ \mathcal{P}^*_x\vert_Y (t)$ does not belong to $Y \subset J(Y)$, otherwise $h^1({\theta_Y} (-t) \otimes \mathcal{P}_x\vert_Y) = 1$. 

       We now compute $h^1(L' (-t) \otimes \mathcal{P}_x\vert_Y \otimes  j {\theta_Y})$ for $j \ge 1$. One has 
$d(L' (-t) \otimes \mathcal{P}_x\vert_Y \otimes  j {\theta_Y}) =  r -2g + jg = r + (j-2) g$. If $j \ge 2$, then $r+(j-2)g \ge r \ge 2g-1$ so that  
$$h^1(L' (-t) \otimes \mathcal{P}_x\vert_Y \otimes  j {\theta_Y}) = 0 \ \forall j \ge 2\, .$$   
Let  $j =1$. Then  $d(L' (-t) \otimes \mathcal{P}_x\vert_Y \otimes j {\theta_Y}) =  r - g$. If $r \ge 3g-1$, then  $r-g \ge 2g-1$ so that $h^1(L' (-t) \otimes \mathcal{P}_x\vert_Y \otimes {\theta_Y}) = 0$. If $2g-1 \le r \le 3g -2$, then $r-g \ge g-1$  so that $h^1(L' (-t) \otimes \mathcal{P}_x\vert_Y \otimes  {\theta_Y}) = 0$ for a general $x \in J(Y)$. Hence we have

\[ h^1(L' (-t) \otimes \mathcal{P}_x\vert_Y \otimes {\theta_Y}) = 
\begin{cases}
 0 &  \mbox{if } j \ge 2\\ 
 0 &   \mbox{if } j=1, r \ge 3g-1\\
 0 &   \mbox{if } j=1, 2g-1 \le r \le 3g-2, \  x \ \mbox{general}. 
\end{cases}
\]

The proposition now follows from the exact sequence \eqref{jpositive}.
Now the cohomology $h^0(V_{r,x} (j \theta)_Y)$ can be computed using Riemann-Roch theorem.

\end{proof}

\section{ACM bundles on the Jacobian surface}  

         In this section, we specialise to $g=2$ i.e. $\overline{J}(Y)$ is the compactified Jacobian of an irreducible projective curve $Y$ of arithmetic genus $2$ with at most nodes as singularities. We show that for each $r \ge 3$, there exits a two dimensional  family of ACM bundles on the surface $\overline{J}(Y)$. In particular, this gives two dimensional  families of ACM bundles on the Jacobian of a smooth curve of genus $2$.
        
\begin{definition} \label{ACMbundles}
      Let $Z$ be an $n$-dimensional projective variety with an ample line bundle $A$. A vector bundle $F$ on $Z$ is called ACM if 
      $$H^i(Z, F \otimes A^j) = 0 \ \forall j \in \mathbb{Z} \ {\rm and} \ 0 < i < n\, .$$
An ACM bundle $F$ is Ulrich if 
$$H^0(Z, F \otimes A^{-1}) = 0 \  {\rm and} \ H^0(Z,F) = r(F) \cdot deg(Z)\, .$$      
\end{definition}

\begin{theorem} \label{PxACMg2}
        Assume that $g=2$ and $\mathcal{P}_x\vert_Y \neq \mathcal{O}_Y, \ \mathcal{P}_x$ being the restriction of the Poincar\'e bundle (defined by \eqref{Px}). Then \\
(1) $h^1(\overline{J}(Y), \mathcal{P}_x(j \theta)) = 0\, , \ \forall j \in \mathbb{Z}\, .$ In particular,  $\mathcal{P}_x$ is an ACM bundle on $\overline{J}(Y)$. \\
(2) $h^0(\overline{J}(Y), \mathcal{P}_x(j \theta)) = 0, \ \forall j \le 0\ ,$  \\
~~ $h^0(\overline{J}(Y), \mathcal{P}_x(j \theta)) = j^2, \ \forall j \ge 0\ .$  \\ 
(3) $h^2(\overline{J}(Y), \mathcal{P}_x(j \theta)) = 0, \ \forall j \ge 0\ ,$\\
~~ $h^2(\overline{J}(Y), \mathcal{P}_x(j \theta)) = j^2, \ \forall j \le 0\ .$
\end{theorem}   
\begin{proof}
(1)  By Theorem \ref{cohomPx}, $H^1(\overline{J}(Y), \mathcal{P}_x(j \theta))= 0$ for $j =0$.
           
          Note that for $g=2$, one has the divisor $\theta = Y$ and $\theta\vert_Y = \omega_Y$. For any integer $j$, there is an exact sequence 
$$0 \longrightarrow \mathcal{P}_x((j-1) \theta) \longrightarrow \mathcal{P}_x(j \theta) \longrightarrow \mathcal{P}_x(j \theta)\vert_Y  \rightarrow 0\, .$$
The associated cohomology exact sequence is  
\begin{equation}\label{p1}
0 \longrightarrow H^0(\overline{J}(Y), \mathcal{P}_x((j-1) \theta)) \longrightarrow H^0(\overline{J}(Y), \mathcal{P}_x(j \theta)) \longrightarrow H^0(\overline{J}(Y), \mathcal{P}_x(j \theta)\vert_Y) 
\end{equation}
\begin{equation}\label{p2}
 \longrightarrow H^1(\overline{J}(Y), \mathcal{P}_x((j-1) \theta)) \longrightarrow H^1(\overline{J}(Y), \mathcal{P}_x(j \theta)) \longrightarrow H^1(\overline{J}(Y), \mathcal{P}_x(j \theta)\vert_Y) 
\end{equation}
\begin{equation}\label{p3}
 \longrightarrow H^2(\overline{J}(Y), \mathcal{P}_x((j-1) \theta)) \longrightarrow H^2(\overline{J}(Y), \mathcal{P}_x(j \theta)) \longrightarrow H^2(\overline{J}(Y), \mathcal{P}_x(j \theta)\vert_Y) = 0\,.
\end{equation}
{\bf Case $j \ge 1$:} \\ 
Since  $d(\mathcal{P}_x)\vert_Y) = 0$, we have $d(\mathcal{P}_x(j \theta))\vert_Y = j g$. Consider the sequence \eqref{p2} for  $j = 1$. Since $P_x\vert_Y  \neq \mathcal{O}_Y,  \mathcal{P}_x(\theta)\vert_Y \neq \omega_Y$. It is a line bundle of degree $2$.  It follows that $h^1(Y, \mathcal{P}_x(\theta)\vert_Y) = 0$. Hence Theorem \ref{cohomPx} and \eqref{p2} give $H^1(\overline{J}(Y), \mathcal{P}_x(\theta)) = 0$.

             Let $j \ge 2$, we prove by induction that $H^1(\overline{J}(Y), \mathcal{P}_x(j \theta)) = 0$. The result is true for $j=1$. By induction, $H^1(\overline{J}(Y), \mathcal{P}_x((j-1) \theta)) = 0$. We have $d(\mathcal{P}_x(j \theta)\vert_Y) \ge 2g$ so that $H^1(\overline{J}(Y), \mathcal{P}_x(j \theta)\vert_Y) = 0$. It follows from the exact sequence \eqref{p2} that $H^1(\overline{J}(Y), \mathcal{P}_x(j \theta)) = 0$.\\
{\bf Case $j \le -1$:} \\             
For $j \le -1$, we prove the result by descending induction. We assume that $H^1(\overline{J}(Y), \mathcal{P}_x(j \theta)) = 0$ and show that $H^1(\overline{J}(Y), \mathcal{P}_x((j-1) \theta)) =0$. From the exact sequences \eqref{p1} and \eqref{p2}, we have the exact sequence 
$$H^0(\overline{J}(Y), \mathcal{P}_x(j \theta)\vert_Y)  \longrightarrow H^1(\overline{J}(Y), \mathcal{P}_x((j-1) \theta)) \longrightarrow H^1(\overline{J}(Y), \mathcal{P}_x(j \theta))$$
Since $d(\mathcal{P}_x(j \theta)\vert_Y) \le -1$ for $j \le -1$, the first term vanishes. The third term vanishes by induction. It follows that the middle term $H^0(\overline{J}(Y), \mathcal{P}_x((j-1) \theta)) = 0$.  \\
(2) For $j \le 0, H^0(\overline{J}(Y), \mathcal{P}_x(j \theta)) \subseteq H^0(\overline{J}(Y), \mathcal{P}_x) = 0$, 

For $j \ge 2$, we have $h^1(\overline{J}(Y), \mathcal{P}_x(j \theta)\vert_Y) =  0$ so that $h^0(\overline{J}(Y), \mathcal{P}_x(j \theta)\vert_Y) = 2j -1$. For $j = 1, h^0(\overline{J}(Y), \mathcal{P}_x(j \theta)\vert_Y) = 1$ (as $\mathcal{P}_x\vert_Y \neq \mathcal{O}_Y$).  

       In view of Part (1), \eqref{p1} and  \eqref{p2}, we have an exact sequence (for all $j$) 
$$0 \longrightarrow H^0(\overline{J}(Y), \mathcal{P}_x((j-1) \theta)) \longrightarrow H^0(\overline{J}(Y), \mathcal{P}_x(j \theta)) \longrightarrow H^0(\overline{J}(Y), \mathcal{P}_x(j \theta)\vert_Y) \to 0\, ,$$
so that $h^0(\overline{J}(Y), \mathcal{P}_x(j \theta)) - h^0(\overline{J}(Y), \mathcal{P}_x((j-1) \theta)) =  h^0(\overline{J}(Y), \mathcal{P}_x(j \theta)\vert_Y) = 2j -1$ for $j \ge 1$. 
 Hence $h^0(\overline{J}(Y), \mathcal{P}_x(j \theta)) = j^2$.\\
 (3) Part (3) can be proved similarly using Part (1), \eqref{p3} and  \eqref{p2}. Alternatively, one can use Part (2) and Serre duality.
\end{proof}

       Let $V_{r, x}, x^* \notin Y$ be a twisted Picard bundle (see Definition \ref{Vrx}) on $\overline{J}(Y)$. 

\begin{theorem}  \label{thm1}
          Let $g =2$. Then the twisted Picard bundles  $V_{r, x}, r \ge 3, x^* \notin Y$ are ACM bundles for $x$ general in $J(Y)$. They are not Ulrich bundles.   
          
          The bundles $E_r \otimes \mathcal{P}_x, r \ge 2, x^* \in Y$ are not ACM bundles.  In particular the Picard bundles are not ACM bundles.       
\end{theorem}        

\begin{proof}
 There is an exact sequence (for any $j$)
\begin{equation} \label{seq}
 0 \to V_{r,x}(j \theta) \to V_{r,x}((j+1) \theta) \to V_{r,x}((j+1) \theta\vert_Y) \to 0\, .
\end{equation} 
         The associated cohomology exact sequence is 
\begin{equation} \label{seq1}
  H^0(V_{r,x}((j+1) \theta)) \to H^0(V_{r,x}((j+1) \theta\vert_Y)) 
  \to H^1(V_{r,x}(j \theta)) 
  \end{equation}
$$  \to H^1(V_{r,x}((j+1) \theta)) \to H^1(V_{r,x}((j+1) \theta\vert_Y)) 
 \to \, .$$
 (a) To prove that $h^1(V_{r,x}(j \theta))= 0$ for $j \le 0\, :$\\
We prove this by descending induction on $j$. By Theorem \ref{chomB}, $h^1(V_{r,x}) = 0$, so the result is true for $j = 0$. Let $j \le -1$. Assume that the result is true for $j+1$, we shall prove that  the result is true for $j$. For $j \le -1$, one has $H^0(V_{r,x}((j+1) \theta)) \subseteq H^0(V_{r,x}) = 0$ (by Theorem \ref{chomB}) and $h^1(V_{r,x}((j+1) \theta)) = 0$ by induction. Then by exact sequence \eqref{seq1} we have 
$$h^1(V_{r,x}(j \theta)) = h^0(V_{r,x}((j+1) \theta\vert_Y)) = 0\, , $$
by Proposition \ref{VL1}. \\
(b) To prove that $h^1(V_{r,x}(j \theta))= 0$ for $j \ge 0$ and for a general $x \in J(Y)$:\\
 This vanishing will be proved by induction on $j$. The result is true for $j=0$. Assuming that 
$h^1(V_{r,x}(j \theta))= 0$, we shall prove that  $h^1(V_{r,x}((j+1) \theta))= 0$. By Proposition \ref{VL2},  $h^1(V_{r,x}((j+1) \theta\vert_Y)) = 0 $ for a general $x \in J(Y)$. It follows from the exact sequence \eqref{seq1} that  $h^1(V_{r,x}((j+1) \theta))= 0$ for all $j \ge 0$.

From parts (a) and (b), we have 
$$ h^1(V_{r,x}(j \theta))= 0  \ \forall  j \in \mathbb{Z}\, , $$
proving that $V_{r,x}, r \ge 3, x^* \notin Y,$ are ACM bundles for a general $x \in J(Y)$.   

          The bundles $V_{r,x}$ are not Ulrich as their zeroth cohomology vanishes by Theorem \ref{chomB}.

          The bundles $E_r \otimes \mathcal{P}_x, r \ge 3, x^* \in Y$ are not ACM bundles as they have non vanishing first cohomology,  by Theorem \ref{chomB}.     
\end{proof}

\begin{remark}
   For any $g$, the bundles $E_r \otimes \mathcal{P}_x, r \ge 2g-1, x^* \in Y$ are not ACM bundles as they have non vanishing middle cohomology by Theorem \ref{chomB}. 
\end{remark}

\begin{section}{Embedding of the Jacobian }

                    Our aim in this section is to define an embedding of the (generalised) Jacobian $J(Y)$ in the moduli space $U^s_Y(n,d)$ of stable vector bundles on $Y$ using the twisted restriction $E_Y$ of a Picard bundle to $Y \subset J(Y)$.  On smooth curves,  Yingchen Li had defined an embedding $\alpha_Y$ of $J(Y)$  into  $U_Y(n,d)$ using a stable vector bundle $E_0$ such that the  bundle of endomorphism $End (E_0)$  has no line bundle direct summands \cite{Li}. The assumption on $End(E_0)$ is crucially used to prove the injectivity of $\alpha_Y$. The  results of \cite{Li} were extended to a nodal curve by me \cite{Bh} under the additional assumption that the pull back of $E_0$ to the normalisation $X$ is stable. We do not know if $E_Y$ has any direct summand which is a line bundle.  Moreover,  the pull back of $E_Y$ to the normalisation $X$ is not semistable.  Therefore, we need several changes in the proofs of \cite{Bh} (or \cite{Li} in the smooth case). 
Our proof is on similar lines as that in \cite{Bh}, \cite{Li}.                    

\subsection{The Embedding $\alpha_Y$ }  \label{alphaembed} \hfill

            We  assume that $r \ge 2g, g \ge 2$. In the notations of subsection \ref{picardsheaves}, we have a universal family $\mathcal{P} \to \overline{J}(Y) \times \overline{J}(Y)$ with the restriction $\mathcal{L} \to \overline{J}(Y) \times Y$ giving the Picard bundle $E_r$ on the first factor $\overline{J}(Y)$. If $p_1, p_2$ are the projections from $\overline{J}(Y) \times \overline{J}(Y) \to \overline{J}(Y)$ to the first and second factors, define 
            $$\mathcal{V} := (p_1^*(E_r) \otimes \mathcal{P}) \vert_{Y \times \overline{J}(Y)}\, ,$$
and  $\mathcal{V}(t) = \mathcal{V} \otimes p_1^*\mathcal{O}_Y(t)$. Let $E_Y := E_r \vert_Y \otimes  \mathcal{O}_Y(t)$. Then for $x \in J(Y)$ (second factor), $\mathcal{V}(t)_{Y \times x} = E_Y(t) \otimes \mathcal{P}_x\vert_Y$. Since $E_Y$ is stable,  so is $E_Y \otimes N$ for any line bundle $N$ on $Y$. Thus $\mathcal{V}(t)\vert_{Y \times J(Y)}$ is a family of stable vector bundles on $Y$, of rank $r_Y= r+1-g$ and degree $d_Y= r+1-2g$, parametrised by $J(Y)$. Hence it gives a morphism
$$\alpha_Y:   J(Y) \longrightarrow U_Y^s(r_Y,d_Y)\, ; \alpha_Y(N) = N \otimes E_Y\, . $$   

\begin{proposition} \label{prop1.6}
  The morphism $\alpha_Y$ is injective. 
\end{proposition}

\begin{proof}
   Suppose that $N_1, N_2 \in J(Y), N_1 \neq N_2$ and $\alpha_Y(N_1) =  \alpha_Y(N_2)$. Then $E_Y \otimes N_1 = E_Y \otimes N_2$ i.e., $E_Y \otimes N = E_Y, N = N_1 \otimes N_2^{-1}$. To prove the injectivity of $\alpha_Y$, it suffices to prove that $E_Y \otimes N \neq E_Y$ for a nontrivial line bundle $N$.  
Recall the exact sequence \eqref{d4a} 
$$0 \to H^0(Y, L)\otimes {\mathcal O}_{Y} \to {E}_Y \to L' \to  0\, ,$$ 
where $L' $ is a line bundle of degree $r+1- 2g, h^0(Y, L) = r- g$.
This implies that $h^0(E_Y) \ge r - g.$

           Tensoring the exact sequence \eqref{d4a} by $N$ and taking the cohomology exact sequence, we get an exact sequence 
$$0 \to  H^0(Y, L) \otimes H^0(Y, N)  \to H^0(Y, E_Y\otimes N) \to H^0(Y, L' \otimes N) \to  $$        
Since $d(N)= 0, N \neq  {\mathcal O}_{Y}$, we have $h^0(Y, N) = 0$. Then 
$h^0(Y, E_Y \otimes N) \le h^0 (Y, L' \otimes N)$. \\
Case 1: Let $r \ge 4g - 2$. Then $d(L'\otimes N) \ge 2g- 1$ so that $h^1(Y, L'\otimes N) = 0$ and  by Riemann-Roch theorem, $h^0(Y, L'\otimes N) = r +2 -3g$. Hence $h^0(Y, E_Y\otimes N) \le r + 2 -3g $. Since $r-g > r+2 -3g$ for $g \ge 2$, we get $h^0(Y, E_Y) > h^0(Y, E_Y\otimes N)$. Therefore $E_Y \ncong E_Y \otimes  N$.\\
Case 2: Let $2g \le r \le 4g-3$. Then $d(L' \otimes N)) \le 2g - 2$. By Clifford's theorem, 
$$h^0(Y, L'\otimes N) \le \frac{r+ 1 - 2g}{2} +1 = \frac{r+3}{2} - g\, .$$ 
Hence $h^0(Y, E_Y\otimes N) \le \frac{r+3}{2} - g < r - g = h^0(Y, E_Y)$ for $r >3$.  It follows that $E_Y \ncong E_Y \otimes  N$.      
\end{proof}

\begin{proposition} \label{embedX}
 The morphism $\alpha_Y$ is an embedding. 
\end{proposition}

\begin{proof}
    The Zariski tangent space to $J(Y)$ at a point $N\in J(Y)$ is Ext$^1(N,N) = H^1(Y, N^*\otimes N) = H^1(Y, \mathcal{O}_Y)$. The Zariski tangent space to $U^s_Y(r_Y, d_Y)$ at a point $E \in U^s_Y(r_Y, d_Y)$ is Ext$^1(E,E) = H^1(Y, End(E))$. 
    The injectivity of $d(\alpha_Y)$, the induced mapping of tangent spaces, can be proved using the determinant map $det : U_Y(r_Y, d_Y)  \to J^{d_Y}(Y)$ as in \cite[Proposition 1.6]{Li} or using deformation theory as in \cite[Remark 1.7]{Li}. 
Since $det \ \circ \ \alpha_Y(N) = det(E_Y) \otimes N^{r_Y}$, upto translation this composite map is the map $J(Y) \to J(Y)$ defined by multiplication by $r_Y$. The latter being an \'etale map, the induced map on tangent spaces is injective and hence the tangent map $d \alpha_Y$ is injective.         
\end{proof}

\subsection{The morphism $\alpha^B_Y$} \hfill

       Let $B$ be a line bundle of degree $b \ge 2g -1$ on the curve $Y$. Tensoring with $B$ defines an isomorphism $t_B: U^s_Y(r_Y, d_Y) \to U^s_Y(r_Y, d^b_Y)$ where $d^b_Y = d_Y+ r_Y b$. Define $\alpha^B_Y= t_B \circ \alpha_Y$, 
$$\alpha^B_Y : J(Y) \to U^s_Y(r_Y, d^b_Y)\, ; \ \alpha^B_Y(N) = N \otimes E_Y\otimes B\, .$$ 
Let $t'_B : J(Y) \to J^B(Y)\, ; \ t'_B (N) = N \otimes B$. Then $\alpha^B_Y$ factors through $t'_B$.
 
         Let $\theta_U$ denote the ample theta line bundle on  $U^s_Y(r_Y, d^b_Y)$. 
We have the greatest common divisor 
$$n := (r_Y, d^b_Y) = (r_Y, d_Y) = (r+1-g, r+1-2g) = (r+1, g)\, .$$

\begin{theorem} \label{theta_Ures}

       Assume that $Y$ is smooth. Let $n = (r+1, g)$. Then
 $$(1) \ \  (\alpha^B_Y)^* \theta_U \ = \ \frac{r_Y^2}{n} \ \theta\, ,$$
 $$(2)  \ \ \alpha^B_Y(Y) \cdot \theta_U \ =  \ \frac{r_Y^2}{n} \ g \, .$$    

\end{theorem}

\begin{proof}
(1) Let 
                   $$\mathcal{N}' \longrightarrow Y \times J^b(Y)$$ 
be the normalised universal family of line bundles of degree $b$ on $Y$ parametrised by $J^b(Y)$,   the Jacobian of line bundles of degree $b$ on $Y$.
We have an isomorphism $t'_B: J(Y) \to J^b(Y): N \mapsto N' = N \otimes B$. Then $\alpha^B_Y= \alpha'_Y \circ t'_B$ where
$$\alpha'_Y: J^b_Y \longrightarrow U^s_Y(r_Y, d^b_Y) \ \ ; \alpha'_Y(N') = p_1^*(E_Y) \otimes N'\, .
$$
 We note that  $\alpha'_Y$ is an embedding as $\alpha^B_Y$ is so. It suffices to prove the theorem for the pair $(J(Y), \alpha^B_Y)$ replaced with the pair $(J^b(Y), \alpha'_Y)$.    
 
           Since the universal family $p_1^*(E_Y)\otimes \mathcal{N}'$ over $Y \times J^b(Y)$ gives the morphism $\alpha'_Y$, we have 
\begin{equation} \label{thetaU1}
(\alpha'_Y)^* \theta_U \cong Det (p_{2 !} (p_1^*(E_Y)\otimes \mathcal{N}') )^{-1}\, ,
\end{equation}    
$Det$ denoting the determinant of cohomology for the family. We shall compute the right hand side of this equation.

       Let $F$ be a vector bundle of rank $r(F) = r_Y/n$ and degree  $d(F) = \frac{-d^b_Y+ r_Y(g-1)}{n}$.  Applying the Grothendieck Riemann-Roch theorem to the morphism $p_2: Y \times J^b(Y) \to J^b(Y)$, one gets
$$ch(p_{2 !} (\ p_1^*(E_Y \otimes F) \otimes \mathcal{N}') \cdot td(J^b(Y))
= p_{2 *} (ch(p_1^*(E_Y \otimes F) \otimes \mathcal{N}') \cdot td(Y \times J^b(Y))\, .$$
Since $td(J^b(Y)) = 1$, we have 
$$
\begin{array}{ll}
ch(p_{2 !} \ (p_1^*(E_Y \otimes F) \otimes \mathcal{N}')) & {}\\
 {} &  =   p_{2 *} (ch(p_1^*(E_Y \otimes F) \otimes \mathcal{N}') \cdot td(Y \times J^b(Y)))\\
 {} & =   p_{2 *} (ch(\mathcal{N}') ch(p_1^*(E_Y \otimes F)) \cdot td(Y \times J^b(Y)))
\end{array}
$$  
By \cite[Lemma 4.4(a)]{Li} (or \cite[pp. 334 - 336]{ACGH}), $ch(\mathcal{N}' )= (1 + b \eta + \gamma - \eta \theta)$ where $\eta$ is the pull-back by $p_1$ of a point in $Y$, $\theta$ is the theta divisor on $J^b(Y)$ and $\gamma$ is a divisor whose poincar\'e duel lies in $H^1(Y, \mathbb{Z}) \otimes H^1(J^b(Y), \mathbb{Z})$,  and one has $\eta \cdot \gamma = 0$. We have $ch(Y \times J^b(Y)) = 1+ (1-g) \eta, ch(p_1^*(E_Y \otimes F)) = r(E_Y \otimes F) + d(E_Y \otimes F) \eta = r_Y^2/ n + r_Y^2(g-1)\eta/ n$. Substituting, we get 
\begin{equation} \label{thetaU2}
ch(p_{2 !} (p_1^*(E_Y \otimes F) \otimes \mathcal{N}') )
= p_{2 *}((1 + b \eta + \gamma - \eta \theta)\cdot (r_Y^2/n + r_Y^2(g-1) \eta/ n) \cdot (1 + (1-g) \eta))\, .
\end{equation}

By \cite[Lemma 4.4(d)]{Li}, $p_{2 *} \gamma = 0$. One has $\eta^2 = 0$. It is easy to see that the only term in the right hand side of equation \eqref{thetaU2} which gives a non-zero term in $H^2(J^b(Y), \mathbb{Z})$ is   
$p_{2 *} ( - (r_Y^2/n) \eta \theta) = -  (r_Y^2/ n) \ \theta$. Hence $Det (p_{2 !} ( p_1^*(E_Y)\otimes \mathcal{N}') )^{-1} \cong  (r_Y^2/ n) \ \theta$. From \eqref{thetaU1}, it follows that  $(\alpha^B_Y)^* \theta_U \ = \ \frac{r_Y^2}{n} \ \theta\, .$
 \\
(2) By Poincar\'e formula, $\theta \cdot (Y) = g$, hence the part (2) follows from the part (1). 

\end{proof}

\subsection{The restrictions of Poincar\'e and Picard bundles} \hfill

 In this subsection, we assume that $(r+1, g) = 1$ i.e., $(r_Y, d^b_Y)= 1$.  
 
 Then $d^b_Y$ and $r_Y$ being coprime there is  a universal bundle on $U^s_Y(r_Y, d^b_Y)$, unique upto a line bundle from $U^s_Y(r_Y, d^b_Y)$. Fix a universal bundle 
$$\mathcal{U} \ \longrightarrow \ Y \times  U^s_Y(r_Y, d^b_Y)\, .$$

We may identify $J(Y)$ with its image in $U^s_Y(r_Y, d^b_Y)$ under $\alpha^B_Y$ for $d(B) \ge 2g -1$. 

\begin{theorem} \label{univbundleres}

       The restriction of the universal bundle $\mathcal{U}$ to $Y \times J(Y)$ is stable for any polarisation of the form $a_1 H_1+ a_2 H_2$, where $a_1, a_2$ are positive integers,  $H_1$ and $H_2$ are any polarisations on $Y$ and $J(Y)$ respectively.
\end{theorem}

\begin{proof}
    The restriction of the universal bundle $\mathcal{U}$ to $Y \times J(Y)$ is (isomorphic to) $(id_Y \times \alpha^B_Y)^* \mathcal{U}$. Let $\mathcal{N}$ denote the universal bundle on $Y \times J(Y)$ normalised by $\mathcal{N}\vert_{t \times J(Y)} = \mathcal{O}_{J(Y)}$. For $N \in J(Y)$,  we have 
$(id_Y \times \alpha^B_Y)^* \mathcal{U}\vert_{Y \times N} = N \otimes E_Y\otimes B = (\mathcal{N} \otimes p_1^*(E_Y \otimes B))\vert_{Y\times N}$. It follows that there exists a line bundle $M$ on $J(Y)$ such that 
\begin{equation} \label{eqA}
(id_Y \times \alpha^B_Y)^* \mathcal{U} \cong \mathcal{N} \otimes p_1^*(E_Y \otimes B) \otimes p_2^*M\, .
\end{equation}      

For $N \in J(Y), (id_Y \times \alpha^B_Y)^* \mathcal{U}\vert_{Y \times N} \cong  N \otimes E_Y\otimes B$ is a stable vector bundle on $Y$ as $E_Y$ is stable on $Y$. For $y \in Y,  (id_Y \times \alpha^B_Y)^* \mathcal{U}\vert_{y \times J(Y)} \cong (\alpha^B_Y)^* \mathcal{U}_y$ where $\mathcal{U}_y = \mathcal{U}\vert_{y \times J(Y)}$ regarded as a vector bundle on $J(Y)$. We have $(\mathcal{N} \otimes p_1^*(E_Y \otimes B) \otimes p_2^*M)\vert_{y \times J(Y)} \cong \mathcal{N}_y \otimes M \otimes I_{r_Y}$ (where $I_{r_Y}$ is a trivial vector bundle of rank $r_Y$ on $J(Y)$), the latter vector bundle is semistable for any polarisation on $J(Y)$. Hence by equation \eqref{eqA} restricted to $y \times J(Y),  (id_Y \times \alpha^B_Y)^* \mathcal{U}\vert_{y \times J(Y)}$ is semistable for any polarisation on $J(Y)$. As in \cite[Lemma 2.2]{BaBrN}, it follows that the restriction $(id_Y \times \alpha^B_Y)^* \mathcal{U}\vert_{Y \times J(Y)}$ is stable for any polarisation of the form $a_1 H_1+ a_2 H_2$, where $a_1, a_2$ are positive integers, $H_1$ and $H_2$ are any polarisations on $Y$ and $J(Y)$ respectively.
\end{proof}

Define the Picard bundle on $U^s_Y(r_Y, d^b_Y)$ by 
$$E_{r_Y, d^b_Y} \ := \ p_{U *} \mathcal{U}$$
where $p_U: Y \times U^s_Y(r_Y, d^b_Y) \to U^s_Y(r_Y, d^b_Y)$ is the projection.

\begin{theorem} \label{picardbundleres}
Assume that $Y$ is a smooth curve and $(r+1, g) =1$. Then
the restriction of the Picard bundle $E_{r_Y, d^b_Y}$ to $J(Y)$ is $\theta$-semistable for $b \ge 2g-1$ and $\theta$-stable for $b \ge 2g$.
\end{theorem} 

\begin{proof}
We have the commutative diagram 
$$
\begin{array}{ccc}
Y \times J(Y) & \longrightarrow & Y \times U^s_Y(r_Y, d^b_Y)\\
{}                    &   {}                     &   {}  \\
p_2 \downarrow     &  {}  &  \downarrow  p_U\\
{}                    &   {}                     &   {}  \\
J(Y)                & \longrightarrow &  U^s_Y(r_Y, d^b_Y)
\end{array}
$$
Hence, by the projection formula,  the restriction 
$$E_{r_Y, d^b_Y}\vert_{J(Y)} := (\alpha^B_Y)^* p_{U *} \mathcal{U} \cong p_{2 *} (id_Y \times \alpha^B_Y)^*  \mathcal{U} \, .$$  
Using equation \eqref{eqA}, 
$$p_{2 *}(id_Y \times \alpha^B_Y)^*  \mathcal{U} \cong p_{2 *} (\mathcal{N} \otimes p_1^*(E_Y \otimes B) \otimes p_2^*M) \cong p_{2 *} (\mathcal{N} \otimes p_1^*(E_Y \otimes B)) \otimes M\, .$$
 By \cite[Proposition 4.1]{BiBrN},  the generalised Picard bundle $p_{2 *} (\mathcal{N} \otimes p_1^*(E_Y \otimes B))$ on $J(Y)$ is $\theta$-semistable for $b \ge 2g-1$ and $\theta$-stable for $b \ge 2g$, hence the result follows. 
\end{proof}

\subsection{Does the morphism $\alpha_Y$ extend to a morphism on $\overline{J}(Y)$?} \hfill
   
   Let $p^S: X^S  \to Y$ be the partial normalisation of the nodal curve $Y$ (with $k$ nodes) obtained by blowing up $s$ number of nodes, say $y_1, \cdots, y_s$, in $Y$. 
   We have a sequence of morphisms 
   $$X^S= X^S_0  \longrightarrow X^S_1 \longrightarrow \cdots \longrightarrow  X^S_s = Y\, ,$$  
where the nodal curve $X^S_j$ is obtained from $X^S_{j-1}$ by identifying smooth points $x_j$ and $z_j$ in $X^S_{j-1}$ to a node $y_j$. The points of $X^S_0$ lying over $x_j$ and $z_j$ will again be denoted by $x_j$ and $z_j$ respectively. 

                     There is a Poincar\'e sheaf ${\mathcal L}^S$ on $\overline{J}(X^S) \times X^S$ normalised by the condition that ${\mathcal L}^S\vert_{\overline{J}(X^S) \times t}$ is the trivial line bundle. It is locally free over $X^S \times x_j$ and $X^S \times z_j, j= 1, \cdots , s$, where $x_j$ and $z_j$  are the points of $X^S$ lying over $y_j$. Let ${\mathcal L}^S_{x_j} = {\mathcal L}^S\vert_{\overline{J}(X^S) \times x_j}$ and ${\mathcal L}^S_{z_j} = {\mathcal L}^S\vert_{\overline{J}(X^S) \times z_j}$, regarded as line bundles on $\overline{J}(X^S)$. Define 
$P_j := \mathbb{P}({\mathcal L}^S_{x_j} \oplus {\mathcal L}^S_{z_j} )$. 
For $j = 0, 1, \cdots, s$, define 
$$\tilde{J}^S(X^S_0) := \overline{J}(X^S)\, ; \tilde{J}^S(X^S_j) := P_1 \times_{\overline{J}(X^S)} \cdots \times_{\overline{J}(X^S)} P_j\, ,$$     
a $\mathbb{P}^1 \times \cdots \times \mathbb{P}^1 (j$-fold product)-bundle on over  $\overline{J}(X^S)$.  Let 
$$\pi_j: \tilde{J}^S(X^S_j) \to \tilde{J}^S(X^S_{j-1})$$ 
be the projection and 
$$\pi^s =  \pi_1 \circ \cdots \circ  \pi_s : \tilde{J}^S(Y) \to \overline{J}(X^S)$$ 
be the composite. 

           Assume that $d \ge 2g$. Let  $E^S_{d,0}$ be the Picard bundle on $\overline{J}(X^S)$. For $j= 1, \cdots, s$, let $E^S_{d, j}$  be the pull back of the Picard bundle on $\overline{J}(X^S_j)$ to $\tilde{J}(X^S_j)$.  As in \cite[Proposition 5.1]{BhP2}, we can show that there is an exact sequence 
$$0 \to E^S_{d, j} \to \pi_j^* E^S_{d, j-1} \to \mathcal{O}_{P_j}(1)  \to 0\, .$$
      For this, in the proof of  \cite[Proposition 5.1]{BhP2}, we only need to replace $\tilde{J}^0(X_k)$ by $\tilde{J}^S(X^S_j)$ and $\tilde{J}^0(X_{k-1})$ by $\tilde{J}^S(X^S_{j-1})$. We omit the details to avoid repetition.  As in \cite[Proposition 5.2]{BhP2}, replacing $X_0$ by $X^S$ etc., we get an exact sequence 
\begin{equation} \label{*2}
0 \to E^S_{d, s} \stackrel{i^S_s}{\longrightarrow} \pi^{s *} E^S_{d, 0} \stackrel{g^S_s}{\longrightarrow} \oplus_j \mathcal{O}_{P_j}(1)  \to 0\, ,
\end{equation}             
here  $\mathcal{O}_{P_j}(1)$ denotes the pull back of $\mathcal{O}_{P_j}(1)$ to $\tilde{J}^S(Y)$ by $\pi^S$. 

        Let $\tilde{\mathcal L}^S: \tilde{J}^S(Y) \times Y$ be the pull back of the Poincar\'e sheaf $\mathcal{L}$. Let $\nu_S: \tilde{J}^S(Y) \times Y \to \tilde{J}^S(Y)$ be the first projection. Recall that $X^S_s = Y$.  As in \cite[Propositions 5.3]{BhP2}, we have an isomorphism 
\begin{equation}\label{*3}
\pi^S_* E^S_{d, s} \cong \nu^S_* [\tilde{\mathcal L}^S(d) ( - \sum_{j=1}^s (x_j + z_j))]\, .
\end{equation}
Let $F^S_{d, s} := \pi^{S  *} \pi^S_* E^S_{d, s}$. Then we also get an exact sequence 
\begin{equation} \label{*4}
0 \rightarrow F^S_{d,s} \rightarrow E^S_{d, s} \rightarrow  \oplus_j \mathcal{O}_{P_j}( - 1) \rightarrow 0\, .
\end{equation}

Define $F^S_Y := F^S_{d,s} \vert_{X^S}(t)$. 
    
\begin{proposition}      \label{XS}
Let $p^S: X^S  \to Y$ be the partial normalisation  obtained by blowing up the nodes $y_1, \cdots, y_s$ in $Y$. Then  $(p^S)^*(E_Y) = F^S_Y \oplus I^s$, where $F^S_Y$ is a stable vector bundle on  $X^S$ contradicting the semistability of $(p^S)^*(E_Y)$ and $I^s$ is the trivial vector bundle of rank $s$. 
\end{proposition}
\begin{proof}
        As in the proof of \cite[Theorem 6.7]{BhP2}, we can show that the vector bundle $F^S_{d, s} = \pi^{S  *} \pi^S_* E^S_{d, s}$ is a stable vector bundle on $\tilde{J}^S(X^S_s)$, in fact this is shown by showing that the restriction of  $F^S_{d, s}$ to the curve $X^S$ embedded in  $\tilde{J}^S(X^S_s)$ is stable. Hence $F^S_Y=  F^S_{d, s}\vert_{X^S} (t)$ is a stable vector bundle on $X^S$ and contradicts the semistability of $E^S_{d, s}\vert_{X^S} (t) = (p^S)^*(E_Y)$. Restricting the exact sequence \eqref{*4} to $X^S$,  we get an exact sequence $0 \to F_Y \to (p^S)^* E_Y \to I_s \to 0$ which can be shown to be a split exact sequence.  This proves the proposition.   	 	 
\end{proof}

\begin{lemma}
   The morphism $\alpha_Y$ does not extend to $\overline{J}(Y)$.
\end{lemma}

\begin{proof}
      If $N$ is not locally free, then $N = p^S_* N'$ where $p^S$ is the partial normalisation obtained by blowing the $s$ number of nodes in $Y$ where $N$ is not locally free. One has 
$(p^S)^*(E_Y) = F^S_Y \oplus I^s$, where $F^S_Y$ is a stable vector bundle contradicting the semistability of $(p^S)^*(E_Y)$ and $I^s$ is the trivial vector bundle of rank $s$ on $X^S$.  Then $(p^S)^*(E_Y)\otimes N' = (F^S_Y\otimes N') \oplus (\oplus_s N')$. Taking $p^S_*$,  by projection formula we get $E_Y \otimes N = p^S_*(F^S_Y\otimes N') \oplus p^S_* (\oplus_s N')$ which is not semistable, so it does not belong to $U^{' s}_Y(r_Y, d_Y)$.
\end{proof}
\end{section}

\end{document}